\RequirePackage{fix-cm}
\documentclass[smallextended]{svjour3}       
\smartqed  
\usepackage{graphicx}
\usepackage{times,a4wide,mathrsfs}
\usepackage{amsmath}
\usepackage{amsfonts}

\newcommand{\C}{\mathbb{C}}

\newcommand{\QQ}{\mathbb{Q}}
\newcommand{\NN}{\mathbb{N}}
\newcommand{\PP}{\mathbb{P}}

\newcommand{\gr}{\hbox{Gr}}

\newcommand{\wt}{\widetilde}
\newcommand{\ima}{\hbox{Im}}

\newcommand{\rom}{\romannumeral}

 \journalname{}
\begin{document}

\title{Correspondences and singular varieties
}


\author{Robert Laterveer 
}


\institute{CNRS - IRMA, Universit\'e de Strasbourg \at
              7 rue Ren\'e Descartes \\
              67084 Strasbourg cedex\\
              France\\
              \email{laterv@math.unistra.fr}           
           }

\date{}

\maketitle

\begin{abstract} What is generally known as the ``Bloch--Srinivas method'' consists of decomposing the diagonal of a smooth projective variety, and then considering the action of correspondences in cohomology. In this note, we observe that this same method can also be extended to singular and quasi--projective varieties. We give two applications of this observation: the first is a version of Mumford's theorem, the second is concerned with the Hodge conjecture for singular varieties.

\keywords{Algebraic cycles \and Chow groups \and Pure motives \and Singular varieties \and Hodge conjecture}
 \subclass{ 14C15 \and  14C25 \and  14C30}
\end{abstract}

\section{Introduction}
\label{intro}

Let $X$ be a smooth complex projective variety. The cycle class maps
  \[ cl^i\colon A^iX_{\QQ}\to H^{2i}(X,\QQ)\]
  from Chow groups to singular cohomology have given rise to some of the most profound and fascinating conjectures in algebraic geometry: the Hodge conjecture (concerning the image of $cl^i$), and the Bloch--Beilinson conjectures (concerning the structure of the kernel of $cl^i$).
  
Since Mumford's work \cite{M}, it is well--known that if the Chow groups $A^iX_{\QQ}$ are ``small'' (in the sense of being supported on some subvariety), then also the singular cohomology groups are small (in the sense that they are supported on some subvariety).
There is for instance the following result:

\begin{theorem}[\cite{J2}] Let $X$ be a smooth projective variety, and suppose $A_0X_{\QQ}$ is supported on a subvariety of dimension $r$. Then 
 the Hodge numbers $h^{p,0}(X)$ are $0$ for $p>r$.
 \end{theorem}

In proving Mumford--type theorems such as this one, the approach of Bloch--Srinivas \cite{BS} has become hugely influential. (The curious reader is invited to look at \cite{Vo} for a fairly comprehensive overview of this circle of ideas, including many exciting subsequent developments it has spawned)
 In brief, the Bloch--Srinivas method consists of decomposing the diagonal, given some input on the level of Chow groups. Then, the action of this decomposition seen as a correspondence turns out to have many consequences on the level of cohomology. 
 
Because of the formalism of correspondences being used, the Bloch--Srinivas method is usually restricted to smooth projective varieties. In this note, on the other hand, we show this method can also be made to work for singular and quasi--projective varieties. The idea is very elementary: if $X$ is a (possibly singular) projective variety of dimension $n$, a correspondence is defined as a cycle $C\in A_n(X\times X)_{\QQ}$. A correspondence defines an action
  \[C_\ast\colon H^j(X,\QQ)\to H_{2n-j}(X,\QQ)\]
  in a natural way. If $C$ is the diagonal, this action is just the natural map (capping with the fundamental class of $X$). It follows that, once we have a decomposition of the diagonal, this will have consequences for
  \[  \ima\bigl( H^j(X,\QQ)\to H_{2n-j}(X,\QQ)\bigr)\ .\]
It turns out that in certain degrees (depending on the dimension of the singular locus), this image is well--understood: it is exactly the subgroup $W_{j-2n}H_{2n-j}(X,\QQ)$ where $W_\ast$ is Deligne's weight filtration (this is proven using intersection homology, cf. lemma \ref{durf}).

We give two applications of this elementary observation. The first is a new version of Mumford's theorem:

\begin{proposition} Let $X$ be a quasi--projective variety of dimension $n$. Suppose
      \[\hbox{Niveau}(A_iX_{\QQ})\le r\ \ \hbox{for\ all\ }i\ ,\]
      and suppose there exists a compactification of $X$ with singular locus of dimension $\le {n+r+1\over 3}$.
 Then
   \[  \gr^k_F W_{-j}H_j(X,\C)=0\ \ \hbox{provided\ }\vert 2k+j\vert >r  \ .\] 
   \end{proposition}

Here, the hypothesis ``Niveau $(A_iX_{\QQ})\le r$'' means that the Chow group $A_iX_{\QQ}$ is supported on an $(i+r)$--dimensional subvariety.
It should be noted that Lewis has obtained several Mumford--type theorems for singular varieties \cite{L2}; his statements and method are somewhat different from the present note.\footnote{The statements in \cite{L2} are considerably sharper than the one obtained in the present note. On the other hand, Lewis gets by by supposing the generalized Hodge conjecture or the Lefschetz standard conjecture hold universally. The aim of the present note is (1) to see how far one could get unconditionally, (2) extending the Bloch--Srinivas argument to singular varieties.}

The second application concerns the Hodge conjecture (as extended to singular and quasi--projective varieties in \cite{J}):

\begin{proposition}\label{hodge} Let $X$ be a quasi--projective variety of dimension $n$, and suppose there exists a compactification with singular locus of dimension $\le {n+4\over 3}$. Suppose
  \[\hbox{Niveau}(A_iX_{\QQ})\le 3\ \ \hbox{for\ all\ }i\le \ell\ .\]
 Then the cycle class map
   \[  A_jX_{\QQ}\ \to\ W_{-2j}H_{2j}(X,\QQ)\cap F^{-j}H_{2j}(X,\C)\]
   is surjective for $j\le \ell+2$. 
 \end{proposition}
 
We present some examples where this can be applied (corollaries \ref{A0} and \ref{A1}).

\section{The Bloch--Srinivas argument}
\label{sec:1}

  \begin{definition} Let $X$ be a quasi--projective variety, and let $A_iX$ denote the Chow group of $i$--dimensional algebraic cycles. We say that 
    \[ \hbox{Niveau}\bigl( A_iX_{\QQ}\bigr)\le r\]
    if there exists a closed (i+r)--dimensional subvariety $Y\subset X$ such that $A_i(X\setminus Y)_{\QQ}=0$.
    \end{definition}

The key to what follows is the following decomposition lemma. This is the Bloch--Srinivas argument \cite{BS}; in his book, Bloch attributes this argument to Colliot--Th\'el\`ene \cite[appendix to lecture 1]{B}.

\begin{lemma}\label{diag} Let $\bar{X}$ be a projective variety of dimension $n$, and $X\subset\bar{X}$ the complement of a closed subvariety $D$. Suppose
  \[\hbox{Niveau}\bigl( A_iX_{\QQ}\bigr)\le r \ \ \ \ \hbox{for\ all\ $i\le \ell$\ .}\]
  Then there is a decomposition of the diagonal
  \[ \Delta=\Delta_0+\Delta_1+\cdots+\Delta_\ell+\Delta^{\ell+1}+\Gamma\ \ \in A_n(\bar{X}\times\bar{X})_{\QQ}\ ,\]
  where $\Delta_j$ is supported on $V_j\times W_j$, $\Delta^{\ell+1}$ is supported on $X\times W_{\ell+1}$, and $V_j\subset\bar{X}$ is of dimension $j+r$, $W_j\subset\bar{X}$ is of dimension $n-j$, and $\Gamma$ is supported on $D\times\bar{X}$.
\end{lemma}

\begin{proof} This is an application of the Bloch--Srinivas method \cite{BS}. 
We use the following two well--known lemmas:

\begin{lemma}\label{lim} Let $X$ and $Z$ be quasi--projective varieties, and suppose $Z$ is irreducible of dimension $n$. Then for any $i$
  \[  A_i(X_{k(Z)})\cong \varinjlim A_{i+n}(X\times U)\ ,\]
  where the limit is taken over opens $U\subset Z$.
  \end{lemma}
  
  \begin{proof} This is usually stated for smooth projective varieties \cite[appendix to Lecture 1]{B}. If one is brave, one goes checking in Quillen's work to see that the proof given in loc. cit. for the smooth case still goes on for singular varieties. Alternatively, take a resolution of singularities and reduce to the smooth case using the ``descent'' exact sequences, and the fact that $\varinjlim$ is an exact functor.
 \end{proof}
 
\begin{lemma}\label{inj} Let $X$ be a quasi--projective variety defined over a field $k$, and let $k\subset K$ be a field extension. Then
  \[A_i(X_k)_{\QQ}\to A_i(X_K)_{\QQ}\]
  is injective.
  \end{lemma}
  
  \begin{proof} This is usually stated for smooth varieties \cite[appendix to Lecture 1]{B}, but the same argument works in general: use lemma \ref{lim} to reduce to the case of a finite extension. For a finite extension, take a resolution of singularities; for smooth varieties, the existence of the norm implies the extension map is a split injection; by descent, the same is true for singular varieties.
 \end{proof}

Now we proceed with the proof of lemma \ref{diag}.
We can reduce to some subfield $k\subset\C$ which is finitely generated over its prime subfield (that is, we may suppose $X$ and $\bar{X}$ and the various subvarieties supporting the $A_iX_{\QQ}$ are defined over $k$). Consider the restriction
  \[ \Delta\in A_n(\bar{X}\times\bar{X})_{\QQ}\to A_n(X\times\bar{X})_{\QQ}\to \varinjlim A_n(X\times U)_{\QQ}=A_0(X_{k(\bar{X})})_{\QQ}\ \]
  (where the $U$ run over opens of $\bar{X}$).
  But
    \[A_0(X_{k(\bar{X})})_{\QQ}\to A_0(X_{\C})_{\QQ}\]
    is injective (lemma \ref{inj}), so $A_0(X_{k(\bar{X})})_{\QQ}$ is supported in dimension $r$. It follows that we get a rational equivalence
    \[ \Delta=\Delta_0+\Delta^1+\Gamma^1\ \ \in A_n(\bar{X}\times\bar{X})_{\QQ}\ ,\]
    where $\Delta_0$ is supported on $V_0\times\bar{X}$, and $\Delta^1$ is supported on $\bar{X}\times W_1$ for some divisor $W_1$, and $\Gamma_1$ is supported on $D\times\bar{X}$. 
    
 If $\ell=0$ we are done. If not, we consider the restriction of the element $\Delta^1$
   \[ \Delta^1\in A_n(\bar{X}\times W_1)_{\QQ}\to A_n(X\times W_1)_{\QQ}\to A_1(X_{k(W_1)})_{\QQ} \ ,\]
   and we use the hypothesis on $A_1(X_{\C})_{\QQ}$.   
    
 Continuing the same process, after $\ell+1$ steps we end up with a decomposition
as desired.   
\end{proof}

Next, we consider correspondences for possibly singular projective varieties:

\begin{definition}\label{corr} Let $X$ be a projective variety of dimension $n$, and $C\in A_n(X\times X)_{\QQ}$. Then $C$ induces an action
  \[ C_\ast\colon\ H^j(X,\QQ)\ \to\ H_{2n-j}(X,\QQ)\ ,\] 
  defined as follows: for $b\in H^j(X,\QQ)$, let
    \[ C_\ast(b):=   (p_2)_\ast \bigl( (p_1)^\ast(b)\cap [C]\bigr)\ \ \in H_{2n-j}(X,\QQ)\ ,\]
  where $p_1$ and $p_2$ denote projections on the first resp. second factor.
  \end{definition}

  This ``correspondence action'' has the following properties (which are well--known, and oft exploited, in the smooth case):    
  
  \begin{lemma}\label{cap} Let $X$ be a projective variety of dimension $n$, and let $\Delta\in A_n(X\times X)$ be the diagonal. Then
    \[ \Delta_\ast b=b\cap[X]\ \ \in H_{2n-j}(X,\QQ)\]
    for any $b\in H^j(X,\QQ)$.
    \end{lemma}
        
  \begin{proof} Let $f\colon\wt{X}\to {X}$ be a resolution of singularities, and let $\wt{\Delta}$ denote the diagonal of $\wt{X}$. Then
     \[   \begin{split} \Delta_\ast(b)&:=(p_2)_\ast\bigl((p_1)^\ast b\cap\Delta\bigr)\\
                                        &=(p_2)_\ast f_\ast \bigl( f^\ast(p_1)^\ast b\cap\wt{\Delta}\bigr)\\
                                        &=f_\ast \wt{\Delta}_\ast(f^\ast b)\\
                                            &=f_\ast(f^\ast b\cap[\wt{X}])=b\cap [{X}].\end{split}\]

  \end{proof}

 \begin{lemma}\label{factor} Let $X$ be a projective variety of dimension $n$, and suppose $C\in A_n(X\times X)_{\QQ}$ is the image of a cycle $c\in A_n(V\times W)_{\QQ}$, for some closed subvarieties $V$ and $W$ in $X$.
   Then there exists a factorization
   \[ 
     \begin{array}[c]{ccc}
        H^j(\widetilde{V}\times\widetilde{W},\QQ)&\stackrel{\cdot[\widetilde{c}]}{\to}&H_{2n-j}(\widetilde{V}\times\widetilde{W},\QQ)  \\
        \uparrow&&\downarrow\\
        H^{j}(\widetilde{V},\QQ)&& H_{2n-j}(\widetilde{W},\QQ)     \\
        \uparrow&&\downarrow\\
        H^{j}({X},\QQ)&\ \ \stackrel{C_\ast}{\to} \ \  &H_{2n-j}({X},\QQ)\ \\
        \end{array}\]         
  (where $\widetilde{V}$ and $\widetilde{W}$ denote resolutions of singularities, and $\widetilde{c}\in A_n(\widetilde{V}\times\widetilde{W})_{\QQ}$ is any cycle mapping to $c$).
  \end{lemma}
  
  \begin{proof} 
  This is a formality. Let
    \[\begin{split}
            &\psi\colon \ \widetilde{V}\to X ,\\
           & \phi\colon \ \widetilde{W}\to X\\  \end{split}\]
            denote the compositions of the resolution morphism with the inclusion morphism. Let $q_1$ and $q_2$ denote the projection from $\widetilde{V}\times\widetilde{W}$ to the first resp. second factor.
            Then for any $b\in H^j(X,\QQ)$,
              \[\begin{split}
                 C_\ast(b):=&(p_2)_\ast \Bigl( (p_1)^\ast(b)\cap [C]\Bigr)\\
                                 =&(p_2)_\ast \Bigl( (p_1)^\ast(b)\cap (\psi\times\phi)_\ast [\widetilde{c}]\Bigr)\\
                                 =&(p_2)_\ast (\psi\times\phi)_\ast\Bigl( (\psi\times\phi)^\ast(p_1)^\ast(b)\cap [\widetilde{c}]\Bigr)\\
                                 =&\phi_\ast(q_2)_\ast\Bigl( (q_1)^\ast\psi^\ast(b)\cap[\widetilde{c}]\Bigr)\ .\\
                                 \end{split}\]  
   \end{proof}
  
  \begin{remark}\label{alsoA} Naturally, definition \ref{corr} extends to other cohomology theories. For instance, if $A^\ast$ denotes the operational Chow cohomology of Fulton--MacPherson \cite{F}, any correspondence $C\in A_n(X\times X)_{\QQ}$ defines an action
    \[ C_\ast\colon\ A^iX_{\QQ}\to A_{n-i}(X)_{\QQ}\ .\]
    The lemmas \ref{cap} and \ref{factor} still hold in this context (indeed, the proofs are the same; they only use formal properties of cohomology/homology).
    \end{remark}

\section{Mumford theorem}


\begin{definition} Let $X$ be a quasi--projective variety. We let $W_\ast$ and $F^\ast$ denote the weight filtration, resp. the Hodge filtration, on cohomology and on homology of $X$ \cite{PS}.
\end{definition}

\begin{proposition} Let $X$ be a quasi--projective variety of dimension $n$. Suppose
      \[\hbox{Niveau}(A_iX_{\QQ})\le r\ \ \hbox{for\ all\ }i\ ,\]
      and suppose there exists a compactification of $X$ with singular locus of dimension $\le {n+r+1\over 3}$.
 Then
   \[  \gr^k_F W_{-j}H_j(X,\C)=0\ \ \hbox{provided\ }\vert 2k+j\vert >r  \ .\] 
   \end{proposition}
 
 This follows from a more precise version:
 
 \begin{proposition}\label{precisemum} Let $X$ be a quasi--projective variety of dimension $n$, and suppose there exists a compactification of $X$ with singular locus of dimension $\le  s$. Suppose
      \[\hbox{Niveau}(A_iX_{\QQ})\le r\ \ \hbox{for\ all\ }i\le \ell\ .\]
Let $j\in[0,n-s]\cap[2s-r,2n]$. Then
  \[  \gr^k_F W_{-j}H_j(X,\C)=0\ \ \hbox{provided\ }\vert 2k+j\vert >r  \ .\] 
\end{proposition} 
      
\begin{proof} Let $\tau\colon X\to\bar{X}$ denote the given compactification, with boundary $D=\bar{X}\setminus X$.
Taking the transpose of the decomposition of lemma \ref{diag}, we obtain a decomposition of the diagonal
   \[ \Delta=\Delta_0+\Delta_1+\cdots+\Delta_{n-r}+\Gamma\ \ \in A_n(\bar{X}\times\bar{X})_{\QQ}\ ,\]
  where $\Delta_i$ is supported on $ V_i\times W_i$, and $V_i$ (resp. $W_i$) is of dimension $j+r$ (resp. $n-j$), and $\Gamma$ is supported on $\bar{X}\times D$.   

\item{\underline{Step 1: $j\le n-s$. }} Let
  \[  a\in \gr^k_F W_{-j} H_j(X,\C)\ ,\]
  with $k$ and $j$ as indicated in the proposition.
  Using strict compatibility of the Hodge filtration, one can find
  \[\bar{a}\in \gr^k_F W_{-j} H_j(\bar{X},\C)  \]
  restricting to $a$ (i.e. $\tau^\ast(\bar{a})=a\in H_{j}(X,\C)$). Applying lemma \ref{durf} below, there exists
  \[  b\in \gr_F^{k+n} H^{2n-j}(\bar{X},\C)\]
  such that
  \[ \bar{a}=b\cap[\bar{X}]\ \ \in \gr^k_F W_{-j} H_j(\bar{X},\C)\ .\]
  In other words, we have
  \[  \bar{a}=\Delta_\ast(b)=(  \Delta_0+\cdots+\Delta_n+\Gamma)_\ast(b)\ \]
  (here we have used lemma \ref{cap}), and hence
  \[ a=\tau^\ast(\bar{a})=\tau^\ast\Bigl( (  \Delta_0+\cdots+\Delta_n+\Gamma)_\ast(b)\Bigr)\ \ \in \gr^k_F W_{-j}H_j(X,\C)\ .\]
  Now, we analyze the actions of these correspondences piece by piece:
  
  First, 
    \[\tau^\ast \Gamma_\ast(b)=0\ .\]
   Indeed, using lemma \ref{factor}, we find that $\Gamma_\ast(b)$ is supported on $D$.
   
 Next, we consider the action of $\Delta_i$. There is a factorization (guaranteed by lemma \ref{factor})
      \[\begin{array}[c]{ccc}
        \cdots&&\\
        \uparrow&&\downarrow\\
       \gr_F^{k+n} H^{2n-j}(\widetilde{V_{i}},\C)&& \gr_F^{k+n-i}H^{2n-2i-j}(\widetilde{W_{i}},\C)   \\
        \uparrow&&\downarrow\\
        \gr_F^{k+n}\gr^W_{2n-j}H^{2n-j}(\bar{X},\C)&\ \ \stackrel{(\Delta_{i})_\ast}{\to} \ \  &\gr_F^{k}W_{-j}H_{j}(\bar{X},\C)\ .\\
        \end{array}\]     
The upper left group (which is just $H^{k+n,n-k-j}(\widetilde{V_i})$) vanishes for $k+n>i+r$ and for $n-k-j>i+r$. The upper right group vanishes for $k+n-i<0$ and for $n-i-j-k<0$. It follows that $(\Delta_i)_\ast(b)$ vanishes unless
  \[  \hbox{both\ }  k+n\hbox{\ and\ }n-k-j\ \in [i,i+r]\ ;\]
  in particular, $(\Delta_i)_\ast(b)$ vanishes under the hypothesis $\vert 2k+j\vert>r$.
  
  \item{\underline{Step 2 : $j\ge 2s-r$}} Let $S$ denote the singular locus of $X$, and $U=X\setminus S$ the non--singular locus.
  We have the exact sequence
    \[  \gr^k_F W_{-j}H_j(S,\C)\to \gr^k_F W_{-j}H_j(X,\C)\to \gr^k W_{-j}H_j(U,\C)\ .\]
  Suppose now $\vert 2k+j\vert>r$. Then the group on the left vanishes for dimension reasons (indeed, suppose for simplicity $S$ is equidimensional of dimension $s$, and let $\widetilde{S}\to S$ be a resolution; then $W_{-j}H_j(S,\C)$ comes from $H^{2s-j}(\widetilde{S},\C)$ which has Hodge level $\le r$). The vanishing of the group on the right follows from lemma \ref{smoothmum} below.

 \begin{lemma}\label{smoothmum}  Let $X$ be a smooth quasi--projective variety, and suppose
      \[\hbox{Niveau}(A_iX_{\QQ})\le r\ \ \hbox{for\ all\ }i\ .\]
      Then
   \[  \gr^k_F W_{-j}H_j(X,\C)=0\ \ \hbox{provided\ }\vert 2k+j\vert >r  \ .\] 
 \end{lemma} 
  
  \begin{proof} Let $\tau\colon X\to\bar{X}$ be a smooth compactification, with boundary $D=\bar{X}\setminus X$. From lemma \ref{diag}, we obtain a decomposition
    \[ \Delta=\Delta_0+\Delta_1+\cdots+\Delta_{n-r}+\Gamma\ \ \in A_n(\bar{X}\times\bar{X})_{\QQ}\ ,\]
  where $\Delta_i$ is supported on $ V_i\times W_i$, and $V_i$ (resp. $W_i$) is of dimension $j+r$ (resp. $n-j$), and $\Gamma$ is supported on $\bar{X}\times D$.   

Given $a\in \gr^k_F W_{-j}H_j(X,\C)$, we can find 
  \[\bar{a}\in \gr^k_F W_{-j}H_j(X,\C)\]
  restricting to $a$. Then we have
    \[a=\tau^\ast(\bar{a}) =\tau^\ast\bigl( (\Delta_0+\cdots+\Delta_{n-r}+\Gamma)_\ast (\bar{a})\bigr)\ \ \in H_j(X,\C)\ .\]
    Just as above, we check that $\tau^\ast\Gamma_\ast(\bar{a})=0$, and that
    \[ (\Delta_i)_\ast(\bar{a})=0\ \ \hbox{provided\ }\vert 2k+j\vert >r\ .\]  
  \end{proof}
  
\begin{lemma}\label{durf} Let $X$ be a projective variety of dimension $n$, and with singular locus of dimension $\le s$. 

\item{(\rom1)}
The natural map
  \[  \gr^W_{j} H^j(X,\QQ)\to W_{j-2n} H_{2n-j}(X,\QQ)\]
  is injective for $j\le n-s$, and surjective for $j\ge n+s$.
  
 \item{(\rom2)} For any $k$, the natural map
   \[ F^k H^j(X,\C)\to F^{k-n} W_{j-2n} H_{2n-j}(X,\C)\]
   is surjective for $j\ge n+s$. 
   
   \item{(\rom3)} The natural map
     \[ H^{2j}(X,\QQ)\cap F^jH^{2j}(X,\C)\ \to\ W_{2j-2n}H_{2n-2j}(X,\QQ)\cap F^{j-n}H_{2n-2j}(X,\C)\]
     is surjective for $j\ge n+s$.
 \end{lemma}
  
 \begin{proof} 
 
 \item{(\rom1)}
 Let $IH^jX$ denote middle--perversity intersection homology with rational coefficients. It follows from work of Durfee \cite{D} that
   \[  IH^jX=\begin{cases}   \gr^W_j H^j(X,\QQ), & j\ge n+s;\\
                                           W_{j-2n} H_{2n-j}(X,\QQ), & j\le n-s\ .
                                           \end{cases}\]
  It is well--known \cite{GM}, \cite{GM2} that the ``Poincar\'e duality'' map factors
    \[  \gr^W_j H^j(X,\QQ) \to IH^jX \to W_{j-2n} H_{2n-j}(X,\QQ)\ .\]
 Moreover, it is known \cite{HS} that the first arrow is injective, and the second arrow surjective.
 
 \item{(\rom2)} The natural map (given by the cap product) is a map of Hodge structures; as such, it is strictly compatible with the Hodge filtration.
 
 \item{(\rom3)} Consider again the factorization
   \[  \gr^W_{2j} H^{2j}(X,\QQ) \stackrel{\cong}{\to} IH^{2j}X \to W_{2j-2n} H_{2n-2j}(X,\QQ)\ .\] 
   The group $IH^{2j}X$ admits a polarized Hodge structure, given by the Hodge--Riemann relations proven in \cite[Theorem 2.2.3]{CM}. This implies (\cite[Corollary 2.24]{Vo}) that a Hodge class in the image comes from a Hodge class in $IH^{2j}X$. But since the left arrow is an isomorphism (and a map of Hodge structures), this Hodge class comes from a Hodge class in $\gr^W_{2j} H^{2j}(X,\QQ)$. 
 \end{proof}

\end{proof}

\begin{remark} The proof of proposition \ref{precisemum} actually yields a slightly more general statement, which is as follows: 
   Let $X$ be a quasi--projective variety of dimension $n$ (no condition on the singular locus), with
      \[\hbox{Niveau}(A_iX_{\QQ})\le r\ \ \hbox{for\ all\ }i\ .\]
 Then
   \[   \ima\Bigl( H^{2n-j}(X,\C)\to H_j(X,\C)\Bigr)\cap \gr^k_F=0\ \ \hbox{provided\ }\vert 2k+j\vert >r  \ .\] 
\end{remark}

\begin{remark} In the smooth case, one can easily obtain Mumford type theorems involving the coniveau filtration rather than the Hodge filtration. Unfortunately, in the singular case I have not been able to obtain such a statement. The problem lies in the use of lemma \ref{durf}: it is not clear to me whether the surjection
  \[ H^j(X,\QQ)\to W_{j-2n}H_{2n-j}(X,\QQ)\]
  respects the coniveau filtration.
  \end{remark}

\section{The Hodge conjecture}

We recall the formulation of the Hodge conjecture that is adapted to singular varieties \cite{J}, \cite{LH}.

\begin{definition}[Hodge conjecture] Let $X$ be a quasi--projective variety, and $j\in\NN$. We say that $HC(X,2j)$ holds if the cycle class map
  \[cl_j\colon A_jX_{\QQ}\ \to\  W_{-2j} H_{2j}(X,\QQ)\cap \gr^{-j}_F W_{-2j}H_{2j}(X,\C)\]
is surjective.
\end{definition}  

\begin{remark} It is known that the Hodge conjecture in degree $2j$ for all smooth projective varieties implies $HC(X,2j)$ for all quasi--projective varieties $X$; this is proven by descent \cite{J}. In particular, for $X$ of dimension $n$ we know that $HC(X,2j)$ is true for $j=0,1,n-1,n$.
\end{remark}

\begin{proposition} Let $X$ be a quasi--projective variety of dimension $n$, and suppose there exists a compactification with singular locus of dimension $\le {n+4\over 3}$. Suppose
  \[\hbox{Niveau}(A_iX_{\QQ})\le 3\ \ \hbox{for\ all\ }i\le \ell\ .\]
 Then $HC(X,2j)$ is true for $j\le \ell+2$. 
 \end{proposition}
 
This follows from a more precise version:

\begin{proposition}\label{precise} Let $X$ be a quasi--projective variety of dimension $n$, and suppose there exists a compactification with singular locus of dimension $\le s$. Suppose
  \[\hbox{Niveau}(A_iX_{\QQ})\le 3\ \ \hbox{for\ all\ }i\le \ell\ .\]
 Then $HC(X,2j)$ is true for $2j\in \bigl([0,n-s,]\cup [2s-2,2n]\bigr)\cap[0,2\ell+4]$. 
 \end{proposition}
 
 \begin{proof} Let $\tau\colon X\to\bar{X}$ denote the given compactification, with boundary $D=\bar{X}\setminus X$. Taking the transpose of the decomposition of lemma \ref{diag}, we obtain a decomposition of the diagonal
   \[ \Delta=\Delta_0+\Delta_1+\cdots+\Delta_\ell+\Delta^{\ell+1}+\Gamma\ \ \in A_n(\bar{X}\times\bar{X})_{\QQ}\ ,\]
  where $\Delta_i$ is supported on $W_i\times V_i$, $\Delta^{\ell+1}$ is supported on $W_{\ell+1}\times\bar{X}$, and $V_i$ (resp. $W_i$) is of dimension $j+3$ (resp. $n-j$), and $\Gamma$ is supported on $\bar{X}\times D$.   

\item{\underline{Step 1: $2j\le \min(n-s,2\ell+4)$. }} Let
  \[a\in W_{-2j} H_{2j}(X,\QQ)\cap \gr^{-j}_F W_{-2j}H_{2j}(X,\C)\]
be a Hodge class. Let $\bar{a}\in W_{-2j} H_{2j}(\bar{X},\QQ)$ be a Hodge class restricting to $a$, i.e. $\tau^\ast(\bar{a})=a$ (to see this exists, one needs to use a resolution of singularities of $\bar{X}$ and the existence of a polarisation on this resolution).
 According to lemma \ref{durf}, there exists a Hodge class
  \[b\in \gr^W_{2n-2j} H^{2n-2j}(\bar{X},\QQ)\cap \gr^{n-j}_F \gr^W_{2n-2j} H^{2n-2j}(\bar{X},\C)\ \]
  such that
  \[ b\cap[X]=\bar{a} \ \ \in H_{2j}(\bar{X},\QQ)\ .\]
  
  It follows that
  \[a=\tau^\ast(\bar{a})=\tau^\ast(\Delta_\ast b)=\tau^\ast\Bigl((\Delta_0+\cdots+\Delta^{\ell+1}+\Gamma)_\ast b\Bigr)\ \ \in H_{2j}({X},\QQ)\ ,\]
  and it remains to analyze the action of each piece in the decomposition:
  
  As for the last piece, obviously
    \[\tau^\ast\Gamma_\ast(b)=0\ ,\]
    as $\Gamma_\ast(b)$ is supported on $D$.
    
    Next, the action of $\Delta^{\ell+1}$. This factors
    \[\begin{array}[c]{ccc}
        \cdots&&\\
        \uparrow&&\downarrow\\
        H^{2n-2j}(\widetilde{W_{\ell+1}},\QQ)\cap F^{n-j}&&\downarrow\\
        \uparrow&&\\
        \gr^W_{2n-2j}H^{2n-2j}(\bar{X},\QQ)\cap F^{n-j}&\ \ \stackrel{(\Delta^{\ell+1})_\ast}{\to} \ \  &H_{2j}(\bar{X},\QQ)\ .\\
        \end{array}\]
      But the group on the left is generated by cycles for $j\le \ell+2$ (this is $HC(\widetilde{W_{\ell+1}},2$)); it follows that
      \[  (\Delta^{\ell+1})_\ast(b)\ \ \in H_{2j}(\bar{X},\QQ)\]
      is a cycle class.
      
    As for the action of $\Delta_i$, this is similar. We have a factorization
      \[\begin{array}[c]{ccc}
        \cdots&&\\
        \uparrow&&\downarrow\\
        H^{2n-2j}(\widetilde{W_{i}},\QQ)\cap F^{n-j}&& H_{2j}(\widetilde{V_{i}},\QQ)\cap F^{j-n}     \\
        \uparrow&&\downarrow\\
        \gr^W_{2n-2j}H^{2n-2j}(\bar{X},\QQ)\cap F^{n-j}&\ \ \stackrel{(\Delta_{i})_\ast}{\to} \ \  &H_{2j}(\bar{X},\QQ)\ .\\
        \end{array}\]   
             
  The upper left group is generated by cycles provided $2n-2j\ge 2\dim\widetilde{W_i}-2=2n-2i-2$, i.e. provided $j\le i+1$. The upper right group is generated by cycles provided $2j\ge 2\dim\widetilde{V_i}-2=2i+4$, i.e. provided $j\ge i+2$. It follows that for any $j$,
  \[ (\Delta_i)_\ast(b)\ \ \in H_{2j}(\bar{X},\QQ)\]
  is a cycle class.

\item{\underline{Step 2: $j\in[ s-1,\ell+2]$. }} Let $U\subset X$ be the complement of the singular locus $S$ of $X$. We have a commutative diagram with exact rows
  \[
 \begin{array}[c]{cccccc}
       A_jS_{\QQ}&\to& A_jX_{\QQ}&\to& A_jU_{\QQ}&\to 0\\
       \downarrow{cl_i}&&\downarrow{cl_i}&&\downarrow{cl_i}&\\
       W_{-2j} H_{2j}(S,\QQ)&\to&W_{-2j} H_{2j}(X,\QQ)&\to& W_{-2j} H_{2j}(U,\QQ)&\to 0\ .
       \end{array}\]
 It follows from lemma \ref{smoothcase} below that for any $j\le \ell+2$ the right vertical map is surjective on Hodge classes. Any Hodge class in $H_{2j}X$ that is supported on $S$ comes from a Hodge class on $S$ (this can be seen by going to a resolution of singularities of $S$). But the left vertical arrow is surjective on Hodge classes provided $j\ge s-1$. 
 
 \begin{lemma}\label{smoothcase} Let $U$ be a smooth quasi--projective variety of dimension $n$, and suppose
  \[\hbox{Niveau}(A_iU_{\QQ})\le 3\ \ \hbox{for\ all\ }i\le \ell\ .\]
 Then $HC(U,2j)$ is true for all $j\le \ell+2$. 
 \end{lemma}
 
 \begin{proof} Let $\tau\colon U\to\bar{U}$ denote a smooth compactification, with boundary $D=\bar{U}\setminus U$.  As above (taking the transpose of the decomposition of lemma \ref{diag}), we obtain a decomposition
   \[ \Delta=\Delta_0+\Delta_1+\cdots+\Delta_\ell+\Delta^{\ell+1}+\Gamma\ \ \in A_n(\bar{U}\times\bar{U})_{\QQ}\ ,\]
  where $\Delta_i$ is supported on $V_i\times W_i$, $\Delta^{\ell+1}$ is supported on $ W_{\ell+1}\times\bar{U}$, and $V_i$ (resp. $W_i$) is of dimension $i+3$ (resp. $n-i$), and $\Gamma$ is supported on $\bar{U}\times D$.   
  
  Let $a\in W_{-2j}H_{2j}(U,\QQ)$ be a Hodge class, where $j\le \ell+2$. Let $\bar{a}\in H_{2j}(\bar{U},\QQ)$ be a Hodge class restricting to $a$. Then
  \[ a=\tau^\ast(\bar{a})=\tau^\ast\bigl(  (\Delta_0)_\ast(\bar{a})+\cdots+(\Delta_\ell)_\ast(\bar{a})+(\Delta^{\ell+1})_\ast(\bar{a})\bigr)\]
  (since obviously $\tau^\ast\Gamma_\ast(\bar{a})=0$).
  But
    \[(\Delta^{\ell+1})_\ast(\bar{a})\ \ \in H_{2j}(\bar{U},\QQ)\]
    is a cycle class, just as above (the action of $\Delta^{\ell+1}$ factors over $H^{2n-2j}(\widetilde{W_{\ell+1}},\QQ)\cap F^{n-j}$, which is generated by cycles for $j\le \ell+2$).
    Likewise, each
    \[(\Delta_{i})_\ast(\bar{a})\ \ \in H_{2j}(\bar{U},\QQ)\]
  is a cycle class (this is the same argument as above).  
 \end{proof}

\end{proof}

\begin{remark} The argument of proposition \ref{precise} actually shows the following weak version of $HC(X,2j)$: let $X$ be projective of dimension $n$, and suppose 
  \[\hbox{Niveau}(A_iX_{\QQ})\le 3\ \ \hbox{for\ all\ }i\le \ell\ .\]
  Then the group
  \[  \ima\bigl(  H^{2n-2j}(X,\QQ)\to H_{2j}(X,\QQ)\bigr)\cap F^{-j}H_{2j}(X,\C)\]
  is generated by algebraic cycles for $j\le \ell+2$.
  \end{remark}
  
 \begin{remark} It seems likely one could likewise prove the generalized Hodge conjecture for quasi--projective varieties (as formulated in \cite[Conjecture 2.4]{L}), in the case where Chow groups have niveau $\le 2$ (extending the smooth projective case \cite{moi}). I have not looked into this.
 \end{remark}
 
 \begin{remark} In \cite{A} and \cite{A2}, Arapura studies the Hodge conjecture for (possibly singular) varieties that have a small Hodge diamond; his approach is somewhat different from the present note.
 \end{remark}

 \begin{corollary}\label{A0} Let $X$ be quasi--projective of dimension $5$, with singular locus of dimension $\le 3$. Suppose
   \[\hbox{Niveau}(A_0X_{\QQ})\le 3\ .\]
   Then $HC(X,4)$ is true.
   \end{corollary}  
   
 In particular, corollary \ref{A0} applies to log $\QQ$--Fano varieties; by a result of Zhang \cite{Z} such varieties are rationally connected, hence $\hbox{Niveau}(A_0X_{\QQ})\le 0$. 
    
\begin{corollary}\label{A1} The Hodge conjecture $HC(X,\ast)$ is completely verified in the following cases:

\item{(\rom1)} $X$ is a cubic of dimension $6$, and with singular locus of dimension $\le 3$;

\item{(\rom2)} $X\subset\PP^8$ is the intersection of a quadric and a cubic, and $X$ has singular locus of dimension $\le 3$.

\end{corollary}

\begin{proof} Since $X$ is in both cases a complete intersection, it suffices to consider $HC(X,j)$ for $j\ge\dim X=6$. The result now follows from proposition 
\ref{hodge}, plus the fact that
  \[ \hbox{Niveau}(A_iX_{\QQ})\le 0\ \ \hbox{for\ }i\le 1\ .\]
In case (\rom1), this statement was proven by Esnault--Levine--Viehweg \cite{ELV}; in case (\rom2) this is proven by Hirschowitz--Iyer \cite{HI}.  
\end{proof}




\begin{acknowledgements}
This note was written while preparing for the Strasbourg ``groupe de travail'' based on the monograph \cite{Vo}. I wish to thank all the participants of this groupe de travail for the very pleasant and stimulating atmosphere.
\end{acknowledgements}



\end{document}